\documentclass[11pt, reqno]{amsart}

\usepackage{amsmath,amsopn,amssymb,amsfonts,stmaryrd,amsthm}
\usepackage{amsrefs}
\usepackage{graphicx}
 \usepackage{hyperref}
\usepackage{enumitem}
\usepackage{mathtools}
\usepackage{verbatim}
\newcommand{\IR}{{\mathbb R}}
\newcommand{\IZ}{{\mathbb Z}}
\newcommand{\IN}{{\mathbb N}}
\newcommand{\IH}{{\mathbb H}}

\renewcommand{\O}{\mathrm{O}}

\theoremstyle{plain}
\newtheorem{thm}{Theorem}[section]

\newtheorem{lem}[thm]{Lemma}
\newtheorem{prop}[thm]{Proposition}

\theoremstyle{definition}
\newtheorem{defn}[thm]{Definition}
\newtheorem*{rem}{Remark}

\newcommand{\pmat}[1]{\left( \smallmatrix #1 \endsmallmatrix \right)}
\newcommand{\mat}[1]{\left( \begin{matrix} #1 \end{matrix} \right)}

\def\lp{\left(}
\def\rp{\right)}
\def\lb{\left[}
\def\rb{\right]}

\def\a{\alpha}
\def\b{\beta}
\def\d{\delta}

\def\l{\lambda}

\def\th{\theta}

\def\e{\epsilon}

\def\t{\tau}

\def\DD{\Delta}
\def\LL{\Lambda}

\def\del{\, \partial}

\def\ddd{\mathrm{d}}

\def\wt{\widetilde}
\def\wh{\widehat}

\newcommand{\andd}{\quad \mbox{ and } \quad}

\setlist[itemize]{noitemsep, topsep=0pt}

\newcommand{\IM}{\mathcal{M}}
\newcommand{\IW}{\mathcal{W}}
\newcommand{\tp}{T}
\newcommand{\up}{{u^\prime}}
\newcommand{\sign}[1]{\mathrm{sign} \left( #1 \right)}
\newcommand{\signb}[1]{\mathrm{sign} \left[ #1 \right]}
\newcommand{\abs}[1]{\left\lvert #1 \right\rvert}
\newcommand{\norm}[1]{\left\lVert #1 \right\rVert}
\newcommand{\prd}{\prod}
\newcommand{\sbst}{\subseteq}
\newcommand{\rr}{\lb r \rb}
\newcommand{\st}{:\ }
\newcommand{\pv}{\mathrm{PV}}
\newcommand{\Si}{{S \cup \{ j \} }}
\newcommand{\ws}{{\wh{S} }}
\renewcommand{\equiv}{\coloneqq }

\title{$r$-tuple Error Functions and Indefinite Theta Series of Higher-Depth}

\author{Caner Nazaroglu}
\address{Enrico Fermi Institute \\ University of Chicago \\
5620 Ellis Av. \\ Chicago Illinois 60637, USA}
\email{cnazaroglu@uchicago.edu}

\thanks{The author acknowledges and thanks the support of NSF grant 1520748. Any opinions, findings, and conclusions or recommendations expressed in this material are those of the author and do not necessarily reflect the views of the National Science Foundation.}
\thanks{Preprint: EFI-16-23, arXiv:1609.01224}
\begin{document}
\begin{abstract}
Theta functions for definite signature lattices constitute a rich source of modular forms. A natural question is then their generalization to indefinite signature lattices. One way to ensure a convergent theta series while keeping the holomorphicity property of definite signature theta series is to restrict the sum over lattice points to a proper subset. Although such series do not generally have the modular properties that a definite signature theta function has, as shown by Zwegers \cite{zwegers2008mock} for signature $(1,n-1)$ lattices, they can be completed to a function that has these modular properties by compromising on the holomorphicity property in a certain way. This construction has recently been generalized to signature $(2,n-2)$ lattices by Alexandrov, Banerjee, Manschot, and Pioline \cite{Alexandrov:2016enp}. A crucial ingredient in this work is the notion of double error functions which naturally lends itself to generalizations. In this work we study the properties of such error functions which we will call $r$-tuple error functions. We then construct an indefinite theta series for signature $(r,n-r)$ lattices and show they can be completed to modular forms by using these $r$-tuple error functions.
\end{abstract}
\maketitle

\tableofcontents

\section{Introduction}
In his seminal work on mock theta functions, Zwegers \cite{zwegers2008mock} gives three closely related constructions for mock modular forms. One of these constructions involves theta series for lattices of signature $(1,n-1)$  extending an earlier work on such lattices by G{\"o}ttsche and Zagier \cite{gottsche1998jacobi}. A natural problem then is to construct similar modular objects out of signature $(r,n-r)$ lattices. Recently, Alexandrov, Banerjee, Manschot, and Pioline \cite{Alexandrov:2016enp} gave such an extension and investigated its properties in detail for the case $r=2$ while suggesting a natural generalization for $r>2$. Further work along these lines after the groundbreaking work of \cite{Alexandrov:2016enp} includes \cite{Bringmann2016} by Bringmann, Kaszian and Rolen which uses and extends the results of \cite{Alexandrov:2016enp} (in particular for $r=3$ case) to work out the modularity properties of a function that arises in the context of Gromov-Witten theory
 and \cite{kudla2016theta} by Kudla which among other things display a relation between indefinite theta functions here and Kudla-Milson theta series \cite{kudla1986theta}.

The main problem for indefinite signature lattices is that the usual $q$-series one constructs for definite signature lattices is no longer a convergent series. One can construct a convergent series by restricting the sum over lattice points to a proper subset of them, however then generically one does not get the modular properties one would get from definite signature lattices. In \cite{zwegers2008mock} holomorphicity properties of such $q$-series are compromised in a specific way to get a modular object. Error functions used in this context are replaced in \cite{Alexandrov:2016enp} by generalized error functions. One of our goals in this paper is to study the properties of generalized error functions which we call $r$-tuple error functions in this work, closely following the methods of \cite{Alexandrov:2016enp} in their study of double error functions. 

A crucial ingredient in the analysis of \cite{Alexandrov:2016enp} is a result by Vign\'eras \cite{Vigneras1977} that shows conditions under which one can deform a theta series for an indefinite signature lattice and obtain a modular object. The deformation is accomplished through a kernel function satisfying a differential equation which we will call Vign\'eras equation. Ordinary error functions used by \cite{zwegers2008mock} and generalized error functions introduced by \cite{Alexandrov:2016enp} and studied here satisfy this equation and hence can be used in the construction of indefinite theta functions. Mere existence of these functions still does not solve the problem entirely though as one should still prove the convergence of the theta series built as such. This is a nontrivial problem and we will give a sufficient set of conditions for convergence again expanding on the methods of \cite{zwegers2008mock} and \cite{Alexandrov:2016enp}.

The outline of this paper is as follows. In Section \ref{sec:Vigneras} we review the results of \cite{Vigneras1977} and set up some notation. Then in Section \ref{sec:generalized_error} we study $r$-tuple error functions proving properties we need for the discussion of indefinite theta functions. This allows us to set up a particular form of indefinite theta series in Section \ref{sec:indef_theta} and give a sufficient set of conditions for its convergence. Finally, in Section \ref{sec:Discussion} we discuss our results and future prospects.

Note: During the course of this study the author heard of an upcoming paper by Zagier and Zwegers on indefinite theta functions of generic signature. Also after this work was completed, a preprint by Westerholt-Raum \cite{raum} has appeared discussing indefinite theta functions over tetrahedral cones. It employs a geometrical approach to discussing asymptotic properties of the kernel $E_r(\IM; u)$ we will define below. Our work instead bases its discussion over generalized complementary error functions as defined in \cite{Alexandrov:2016enp} and proves its properties for general case through their integral definitions. In particular, the decomposition of the kernel $E_r (\IM; u)$ in terms of generalized complementary error functions $M_r (\IM; u)$ (see Proposition \ref{prop:M_decomposition_of_E} and \ref{prop:boosted_decomposition}) is what is used to establish convergence properties for theta functions.

\section{Vign\'eras' Theorem and Theta Series for Indefinite Signature Lattices}\label{sec:Vigneras}

The main technical tool we will use for establishing modularity properties is Vign\'eras' theorem which we are going to review here. First we set up some notation mainly following that of \cite{Alexandrov:2016enp}. Let $\LL$ be an $n$-dimensional lattice ($n \in \IN$) endowed with an integral bilinear form $B(m,k) = m^\tp A \, k$ for $m,k \in \LL$ (and an associated quadratic form $Q(k) = k^\tp A \, k$) which we also linearly extend to $\IR^n \cong \LL \otimes \IR$. Assume that the bilinear form has signature $(r,n-r)$ where $n \geq r$ and $r\in \IN$ denotes the number of positive eigenvalues.  We will also use the notation $\del_x f(x) \equiv \lp \del_{x_1} f, \ldots, \del_{x_s} f \rp^\tp$ for $x = (x_1, \ldots, x_s)^\tp$. Lastly, we define theta series with kernel $\phi$ by (for $\l \in \IZ$, $\mu \in \LL^* / \LL$ where $\LL^*$ is the dual lattice, $\tau \equiv \tau_1 + i \tau_2 \in \IH$ for $\t_1\in \IR, \t_2 \in \IR^+$, $q \equiv e^{2\pi i \tau}$, $b,c \in \IR^n$ and $p \in \LL$ which is a characteristic vector satisfying $Q(k) + B(k,p) \in 2 \IZ$ for all $k \in \LL$)
\begin{equation}
\theta_\mu \lb \phi, \l \rb (\t, b,c) \equiv
\t_2^{-\l/2} \sum_{k \in \LL + \mu + p/2} e^{\pi i B(k,p)} \, \phi(\sqrt{2 \t_2} (k+b)) \,
q^{-Q(k+b)/2} \, e^{2 \pi i B(c,k+b/2)}.
\end{equation}
If $\phi(x) e^{\pi Q(x) / 2} \in L^1 (\IR^n)$ the absolute convergence of the sum is ensured. Now we can state  Vign\'eras's theorem:
\begin{thm}[Vign\'eras \cite{Vigneras1977}]\label{thm:vigneras}
If for any degree $\leq 2$ polynomial $R(x)$ and order $\leq 2$ differential operator $D(x)$, the functions $\phi(x) e^{\pi Q(x) / 2}$, $D(x) \lb \phi(x) e^{\pi Q(x) / 2} \rb$ and $R(x) \phi(x) e^{\pi Q(x) / 2}$ are in $L^1 (\IR^n) \cap L^2 (\IR^n)$ and if the kernel $\phi (x)$ satisfies the Vign\'eras equation
\begin{equation}
\lb B^{-1} (\del_x, \del_x ) + 2 \pi x^\tp \del_x \rb \phi(x) =  2 \pi \l \phi (x)
\end{equation}
where $B^{-1} (x,y) \equiv x^\tp A^{-1} y$, the theta function $\theta_\mu \lb \phi, \l \rb (\t, b,c) $ transforms like a Jacobi form of weight $(\l + n/2, 0)$. That is we have:
\begin{itemize}
\item $\displaystyle\theta_\mu \lb \phi, \l \rb (-1/\t, c,-b)
 = i^{\l + r} \frac{(- i \t)^{\l+n/2}}{\sqrt{\abs{\LL^*/ \LL}}}
e^{\pi i Q(p) /2} \sum_{\nu \in \LL^* / \LL} e^{2 \pi i B(\mu, \nu)}   
\theta_\nu \lb \phi, \l \rb (\t, b,c), 
 $
 \item 
 $
 \theta_\mu \lb \phi, \l \rb (\t+1, b,c+b) = e^{- \pi i Q(\mu + p/2)} 
 \theta_\mu \lb \phi, \l \rb (\t, b,c) , 
 $
 \item 
 $
 \theta_\mu \lb \phi, \l \rb (\t, b+k, c)  = (-1)^{B(k,p)} e^{-\pi i B(c,k)}
 \theta_\mu \lb \phi, \l \rb (\t, b,c) 
 $ for any $k \in \LL$,
 \item 
 $
 \theta_\mu \lb \phi, \l \rb (\t, b, c+k)  = (-1)^{B(k,p)} e^{\pi i B(b,k)}
 \theta_\mu \lb \phi, \l \rb (\t, b,c) 
 $ for any $k \in \LL$.
\end{itemize}
When we state holomorphicity in $\t$ and $z \equiv b \t - c$, we mean holomorphicity of the function $\wt{\th}_\mu \lb \phi, \l \rb (\t, z) \equiv  e^{\pi i B(b,z)}
\th_\mu \lb \phi, \l \rb (\t, b,c)$.
\end{thm}
When the kernel asymptotes to a locally polynomial and homogeneous function of degree $\l$ one can recover it from its shadow $\psi = \frac{i}{4} (x \del_x - \l) \phi$ and its asymptotic behavior. See \cite{Alexandrov:2016enp} for further details.

\section{Generalized Error Functions}\label{sec:generalized_error}
In this section we will study a natural generalization of double error functions as suggested by \cite{Alexandrov:2016enp} and prove the properties we need to define indefinite theta functions out of them. In this section and in the rest of this paper we will use the following notation: For an $s \times t$ matrix $G$, $G_{S,T}$ where $S \sbst [s]$ and $T \sbst [t]$ will mean the matrix $G$ restricted to rows and columns corresponding to subsets $S$ and $T$, respectively. If $G$ is a column vector we will drop $T$ from this notation if $T = \{ 1 \}$. Also for a column matrix $x = \lp x_1, \ldots, x_s \rp^\tp$ we will use $\prod x \equiv \prod_{j=1}^s x_j$ and $\sign{x} \equiv \prod_{j=1}^s \sign{x_j}$.

\subsection{\texorpdfstring{$r$}{r}-tuple Error Functions}\label{sec:error_fnc}
\begin{defn}\label{def:Mr}
Let $m^{(1)}, \ldots, m^{(r)} \in \IR^{r \times 1}$ be a collection of $r$ non-degenerate column vectors and $w^{(1)}, \ldots, w^{(r)} \in \IR^{r \times 1}$ be the corresponding dual basis (with respect to the Euclidean norm so that they satisfy $w^{(j_1) \tp} \, m^{(j_2)} = \d^{j_1j_2}$). Let us also define $\IM \in \IR^{r \times r} $ by $\IM = \lp  m^{(1)} \ldots m^{(r)}  \rp$ and $\IW \in \IR^{r \times r} $ by $\IW = \lp  w^{(1)} \ldots w^{(r)}  \rp$ so that $\IM^{-\tp} = \IW$. Finally let $u \in \IR^{r \times 1}$, where $u = (u_1, \ldots, u_r)^\tp$ is such that $u^\tp \, w^{(j)} \neq 0$ for all $j = 1, \ldots r$. Then we define `complementary r-tuple error function' $M_r ( \IM; u)$ using the following absolutely convergent integral:
\begin{equation}\label{eq:def_Mr}
M_r (\IM; u) \equiv \lp \frac{i}{\pi} \rp^r  \lvert \det \IM \rvert^{-1}  \int\displaylimits_{\IR^r - i u} \ddd^r z \,
\frac{e^{-\pi z^\tp z - 2 \pi i z^\tp u}}{\prd \lp \IM^{-1} z \rp},
\end{equation}
where the integration variable is represented as a column matrix $z = (z_1, \ldots, z_r)^\tp$.
\end{defn}

\begin{defn}\label{def:Er}
Let $m^{(1)}, \ldots, m^{(r)} \in \IR^{r \times 1}$ be a collection of $r$ non-degenerate column vectors (where we use $\IM \equiv \lp  m^{(1)} \ldots m^{(r)}  \rp$ as in Definition \ref{def:Mr}) and let  $u = (u_1, \ldots, u_r)^\tp \in \IR^{r \times 1}$. We then define `r-tuple error function' $E_r ( \IM; u)$ as
\begin{equation}\label{eq:def_Er}
E_r (\IM; u) \equiv   \int\displaylimits_{\IR^r} \ddd^r  \up \,  e^{- \pi (u - \up)^\tp (u - \up)} \,
\sign{\IM^\tp \up}.
\end{equation}
\end{defn}
Note that $E_r (\IM; u)$ is a $\mathcal{C}^\infty$ function of $u$ for any non-degenerate $\IM$.\footnote{It is useful to compare our definitions to those of \cite{Alexandrov:2016enp}. $M_1(1;u)$ here is simply equal to $M_1(u) = - \sign{u} \, \mathrm{erfc} \lp \lvert u \rvert \sqrt{\pi} \rp$ there, $M_2 \lp \pmat{1 & - \a \\ 0 & 1}^{-\tp} ; \pmat{u_1 \\ u_2}  \rp$ here is equal to the double error function $M_2 (\a; u_1, u_2)$ of \cite{Alexandrov:2016enp} and  $M_2 \lp \pmat{1 & 1 \\ -\a & -\b}^{-\tp} ; \pmat{u_1 \\ u_2}  \rp$ here is equal to $- M_2 \lp \a, \b; u_1, u_2 \rp \, \sign{\a-\b}$ there.}

\begin{prop}\label{prop:Mr_basic_property}
\hfill
\begin{enumerate}[label=(\alph*)]
\item $M_r$ and $E_r$ are invariant under permutations of $m^{(j)}$'s. In other words, for any $r \times r$ permutation matrix $P$ we have $M_r (\IM P; u) = M_r (\IM; u)$ and $E_r (\IM P; u) = E_r (\IM; u)$. Moreover, $M_r$ and $E_r$ do not change under independent positive scalings of $m^{(j)}$'s and change their sign whenever one of $m^{(j)}$'s changes its sign; in other words, for any diagonal $r \times r$ diagonal matrix $D$ all of whose diagonal entries are non-zero real numbers we have $M_r (\IM D; u) = \sign{\det D} \, M_r (\IM; u)$ and $E_r (\IM D; u) = \sign{\det D} \, E_r (\IM; u)$.

\item $M_r$ and $E_r$ are invariant under orthogonal transformations, that is, for any $\LL \in \O (r; \IR)$ we have
\begin{equation}
M_r (\LL \IM; \LL u) = M_r (\IM; u)
\andd
E_r (\LL \IM; \LL u) = E_r (\IM; u) .
\end{equation}

\item If $\IM = \pmat{\IM_s^{(1)}  &  0 \\  0  & \IM_{r-s}^{(2)}}$ is of block diagonal form then
\begin{align}
M_r (\IM; u) &= M_s (\IM_s^{(1)}; u_{[1,s]})  \, M_{r-s} (\IM_{r-s}^{(2)}; u_{[s+1,r]})    \\
&  \andd  \notag \\
E_r (\IM; u) &= E_s (\IM_s^{(1)}; u_{[1,s]})  \, E_{r-s} (\IM_{r-s}^{(2)}; u_{[s+1,r]})
\end{align}
where $u_{[j_1,j_2]} \equiv (u_{j_1}, \ldots, u_{j_2})^\tp$. Note that whenever $m^{(j)}$'s split into two sets spanning orthogonal subspaces we have a similar factorization property using parts (a) and (b) of this proposition since then $\IM$ can be brought into a block diagonal form using $\O (r; \IR)$ transformations and permutations.
\end{enumerate}
\end{prop}
\begin{proof}
All of the statements trivially follow from Definitions \ref{def:Mr} and \ref{def:Er}.
\end{proof}

Before proceeding any further let us introduce some notation.
\begin{itemize}
\item For each $S \sbst \rr$ consider the subspace spanned by $\{ m^{(j)} \st j \in S \}$ and pick an orthonormal basis for it, $b_1^{(S)}, \ldots, b_{\abs{S}}^{(S)}$, where we will use the standard basis $b_1 = (1,0,\ldots,0)^\tp$,  $b_2 = (0,1,0,\ldots,0)^\tp$, $\ldots$ , $b_r = (0,0,\ldots,0,1)^\tp$ for $S = \rr$.

Now for any $S \sbst S' \sbst \rr$, form matrices $Q_{S,S'} \in \IR^{\abs{S} \times \abs{S'}}$ whose rows are the components of $b_1^{(S)}, \ldots, b_{\abs{S}}^{(S)}$ in the basis $b_1^{(S')}, \ldots, b_{\abs{S'}}^{(S')}$, in other words $\lp Q_{S,S'} \rp_{j_1j_2} = b_{j_1}^{(S)\tp} b_{j_2}^{(S')}$. We also will use $Q_S \equiv Q_{S,\rr}$. Essentially, these matrices will form the projectors to subspaces  $\langle m^{(j)} \st j \in S \rangle$. Choosing different orthonormal bases correspond to transforming $Q_{S,S'} \to \Lambda_{\abs{S}} Q_{S,S'}  \Lambda^\tp_{\abs{S'}}$ for $S' \neq \rr$ and $Q_{S} \to \Lambda_{\abs{S}} Q_{S}$ where $\Lambda_n \in \O (n; \IR)$.

Let us state now a couple of properties for future reference:
\begin{enumerate}
\item $Q_{S,S'} \, Q_{S,S'}^\tp = I_{\abs{S}}$ for any $S \sbst S' \sbst \rr$.
\item $Q_S^\tp Q_S m^{(j)} = m^{(j)}$ for $j \in S$.
\item $Q_S w^{(j)} = 0$ for $j \in \rr / S$.
\item $Q_{S,S'} Q_{S',S''} = Q_{S,S''}$ for any $S \sbst S' \sbst S'' \sbst \rr$.
\end{enumerate}

\item Similarly, for each $S \sbst \rr$ consider the subspace spanned by $\{ w^{(j)} \st j \in S \}$ and pick an orthonormal basis for it, $c_1^{(S)}, \ldots, c_{\abs{S}}^{(S)}$, where again we will use the standard basis $c_1 = (1,0,\ldots,0)^\tp$,   $c_2 = (0,1,0,\ldots,0)^\tp$, $\ldots$, $c_r = (0,0,\ldots,0,1)^\tp$ for $S = \rr$.

For any $S \sbst S' \sbst \rr$ we form matrices $P_{S,S'} \in \IR^{\abs{S} \times \abs{S'}}$ whose rows are the components of $c_1^{(S)}, \ldots, c_{\abs{S}}^{(S)}$ in the basis $c_1^{(S')}, \ldots, c_{\abs{S'}}^{(S')}$, or in other words $\lp P_{S,S'} \rp_{j_1j_2} = c_{j_1}^{(S)\tp} c_{j_2}^{(S')}$. We also will use $P_S \equiv P_{S,\rr}$. Choosing different orthonormal bases correspond to transforming $P_{S,S'} \to \Lambda_{\abs{S}} P_{S,S'}  \Lambda^\tp_{\abs{S'}}$ for $S' \neq \rr$ and $P_{S} \to \Lambda_{\abs{S}} P_{S}$ where $\Lambda_n \in \O (n; \IR)$.

These matrices satisfy:
\begin{enumerate}
\item $P_{S,S'} \, P_{S,S'}^\tp = I_{\abs{S}}$ for any $S \sbst S' \sbst \rr$.
\item $P_S^\tp P_S w^{(j)} = w^{(j)}$ for $j \in S$.
\item $P_S m^{(j)} = 0$ for $j \in \rr / S$.
\item $P_{S,S'} P_{S',S''} = P_{S,S''}$ for any $S \sbst S' \sbst S'' \sbst \rr$.
\item $\mat{ Q_S \\P_{\rr /S} }  \in \O (r, \IR) $ for any $S \sbst \rr$.
\end{enumerate}

\item Let $\IM_S$ denote the matrix $\IM_S = \lp m^{(j_1)} \ m^{(j_2)} \ldots m^{(j_{\abs{S}})} \rp$ where $j_1,j_2,\ldots,j_{\abs{S}}\in S \sbst \rr$ and  $j_1 < j_2 < \ldots < j_{\abs{S}}$. We will also use $\IW_S$ for similarly constructed matrices out of $w^{(j)}$'s. Note that $\IW_S^\tp \, \IM_S = I_{\abs{S}}$ and moreover since $Q_S^\tp Q_S \IM_S = \IM_S$ and $P_S^\tp P_S \IW_S = \IW_S$ we have $\lp Q_S \IM_S \rp^{-1} = \IW_S^\tp Q_S^\tp$ and  $\lp P_S \IM_S \rp^{-1} = \IW_S^\tp P_S^\tp$.

\end{itemize}

\begin{prop}\label{prop:Mr_discontunity}
For any nonsingular $\IM \in \IR^{r \times r}$ and $u \in \IR^{r \times 1}$ away from the loci $w^{(j) \tp} u = 0$, the function 
$M_r (\IM; u)$ is a real valued $\mathcal{C}^\infty$ function. Its discontinuity as $w^{(j) \tp} u \to 0$ for all $j \in \rr / S$ is given by
\begin{equation}\label{eq:Mr_discontunity}
M_r (\IM; u) \to (-1)^{r - \abs{S}} \,  \sign{\IW^\tp_{\rr / S} u } \,
							M_{\abs{S}} \lp Q_S \IM_S; Q_S u \rp.
\end{equation}
\end{prop}
\begin{proof}
We start by defining variables $v_j = \frac{w^{(j) \tp} z }{w^{(j) \tp} u}+i $. The Jacobian factor associated with this change of variables is $\abs{\frac{\del v}{\del z}} = \frac{\abs{\det \IM}^{-1}}{\prod\displaylimits_{j=1}^r \abs{w^{(j)\tp}u }}$.
Defining $\wt{v} (v, u, \IM) \equiv \mat{ v_1 w^{(1)\tp} u \\ \vdots \\ v_r w^{(r)\tp} u}$ we can rewrite $M_r (\IM; u)$ as
\begin{equation}
\lp \frac{i}{\pi} \rp^r  \sign{\IW^\tp_{\rr / S} u }  \sign{\IW^\tp_{S} u } \, e^{- \pi u^\tp u} \int\displaylimits_{\IR^r} \ddd^r v \, \frac{e^{-\pi \wt{v}(v,u,\IM)^\tp \IM^\tp \IM \, \wt{v}(v,u,\IM)  }}{\prod \displaylimits_{j=1}^r (v_j - i)} 
.
\end{equation}
As $w^{(j) \tp} u \to 0 $ for $j \in \rr / S$, the components of $\wt{v}(v, u, \IM)$ corresponding to $j \in \rr / S$ go to zero and $\IM \, \wt{v}(v, u, \IM) \to \sum_{j \in S} m^{(j)} v_j w^{(j) \tp} u$ and hence
\begin{equation}
\wt{v}(v, u, \IM)^\tp \IM^\tp \IM \, \wt{v}(v, u, \IM) 
\to \wt{v}_S(v, u, \IM)^\tp \IM_S^\tp \IM_S \, \wt{v}_S(v, u, \IM),
\end{equation}
where we should also note that $\IM_S^\tp \IM_S = \IM_S^\tp Q_S^\tp Q_S \IM_S$.
Also for $\IW^\tp_{\rr / S} u = 0$ we have $Q_S^\tp Q_S u = u$ which in turn gives $\IW_S^\tp u  = \IW_S^T Q_S^\tp Q_S u = (Q_S \IM_S)^{-1} Q_S u$, $u^\tp u = u^\tp  Q_S^\tp Q_S u$ and implies $\wt{v}_S(v, u, \IM) = \wt{v} (v_S, Q_S u, Q_S \IM_S)$. Then we can combine the factor $\lp \frac{i}{\pi} \rp^{\abs{S}}   \sign{\IW^\tp_{S} u } \, e^{- \pi u^\tp u} $ with integrals over $v_S$ to obtain the $M_{\abs{S}} \lp Q_S \IM_S; Q_S u \rp$ part. Finally, in the limit $\IW^\tp_{\rr / S} u \to 0$, remaining $v_{\rr / S}$ integrals give $(i \pi)^{r - \abs{S}}$ and we obtain the discontinuity described in \eqref{eq:Mr_discontunity}.
\end{proof}

\begin{rem}
For large $u$, $E_r (\IM; u)$ is locally constant as $E_r (\IM; u) \sim \sign{\IM^\tp u}$ whereas $M_r (\IM; u)$ is exponentially suppressed as $M_r (\IM; u) \sim \displaystyle \frac{(-1)^r}{\pi^r} \abs{\det \IM}^{-1} \frac{ e^{- \pi u^\tp u} }{\prd \lp \IW^\tp u \rp}$. The asymptotic behavior of $E_r (\IM; u)$ is obvious from its definition in equation \eqref{eq:def_Er} whereas the asymptotic behavior of $M_r (\IM; u)$ can be deduced from a saddle point approximation. We will make the asymptotic behavior of both functions more precise in our discussion.
\end{rem}

\begin{lem}\label{prop:M_E_derivative}
First derivatives of $M_r (\IM; u)$ and $E_r (\IM; u)$ with respect to $u$ are given by:
\begin{equation}
w^{(j) \tp} \del_u M_r (\IM; u) = \frac{2}{\norm{m^{(j)}}  }
e^{- \pi u^\tp Q_{\{ j \} }^\tp Q_{\{ j \} } u    } 
M_{r-1} (P_{\rr / \{ j \} }  \IM_{\rr / \{ j \} }; P_{\rr / \{ j \} } u),
\end{equation}
\begin{equation}
w^{(j) \tp} \del_u E_r (\IM; u) = \frac{2}{\norm{m^{(j)}}  }
e^{- \pi u^\tp Q_{\{ j \} }^\tp Q_{\{ j \} } u    } 
E_{r-1} (P_{\rr / \{ j \} }  \IM_{\rr / \{ j \} }; P_{\rr / \{ j \} } u).
\end{equation}
\end{lem}
\begin{proof}
We start with  $M_r (\IM; u)$ and its definition in terms of the integral $M_r (\IM; u) = \displaystyle \lp \frac{i}{\pi} \rp^r  \abs{ \det \IM }^{-1}  \int\displaylimits_{\IR^r - i u} \ddd^r z \,
\frac{e^{-\pi z^\tp z - 2 \pi i z^\tp u}}{\prd \lp \IW^\tp z \rp}$. The derivative $\del_u$ acting on the integral limits gives vanishing contribution because of the exponential suppression, then we simply act it on the integrand to find
\begin{equation}
w^{(j) \tp} \del_u M_r (\IM; u) = \lp \frac{i}{\pi} \rp^{r-1} 2 \abs{ \det \IM }^{-1}
\int\displaylimits_{\IR^r - i u} \ddd^r z \,
\frac{e^{-\pi z^\tp z - 2 \pi i z^\tp u}}{\prd \lp  \IW^\tp_{\rr / \{ j \} } z \rp}.
\end{equation}
Since $\mat{ Q_{\{ j \}  } \\P_{\rr / \{ j \}} }  \in \O (r, \IR)$ and $P_{\rr / \{ j \}}^\tp P_{\rr / \{ j \}}  \IW_{\rr / \{ j \}}  = \IW_{\rr / \{ j \}} $ we can rewrite this as 
\begin{align}
w^{(j) \tp} \del_u M_r (\IM; u) &= \lp \frac{i}{\pi} \rp^{r-1} 2 \abs{ \det \IM }^{-1} 
\int\displaylimits_{\IR^r - i u} \ddd^r z \,
\frac{e^{-\pi z^\tp P_{\rr / \{ j \}}^\tp P_{\rr / \{ j \}}  z - 2 \pi i z^\tp P_{\rr / \{ j \}}^\tp P_{\rr / \{ j \}}  u}}{\prd \lp  \IW^\tp_{\rr / \{ j \} } P_{\rr / \{ j \}}^\tp P_{\rr / \{ j \}}  z \rp} 
\notag \\
& \qquad \times e^{- \pi u^\tp Q_{\{ j \} }^\tp Q_{\{ j \} } u    } 
  e^{-\pi (z+i u)^\tp Q_{ \{ j \}}^\tp Q_{\{ j \}}  (z+i u)}.
\end{align}
Performing a change of variables $\wt{z} = P_{\rr / \{ j \}} z$ and $z_0  = Q_{\{ j \}} (z + i u )$ and taking the integral over $z_0$ we get
\begin{equation}
w^{(j) \tp} \del_u M_r (\IM; u) = 2  e^{- \pi u^\tp Q_{\{ j \} }^\tp Q_{\{ j \} } u    } 
\frac{\abs{\det P_{\rr / \{ j \}} \IM_{\rr / \{ j \}}}   }{ \abs{\det \IM}}
M_{r-1} (P_{\rr / \{ j \} }  \IM_{\rr / \{ j \} }; P_{\rr / \{ j \} } u) .
\end{equation}
Now note that 
\begin{equation}
\abs{\det \IM} =  \abs{ \mat{ Q_{\{ j \}  } \\P_{\rr / \{ j \}} }  \IM  }
= \abs{\mat{  Q_{\{ j \} } \IM_{\rr / \{ j \} }    &  Q_{\{ j \} } m^{(j)}   \\
P_{\rr / \{ j \}} \IM_{\rr / \{ j \} }    & P_{\rr / \{ j \}}  m^{(j)}      }}.
\end{equation}
Since $P_{\rr / \{ j \}}  m^{(j)}  = 0$ this simply reduces to
\begin{equation}
\abs{\det \IM} = \abs{Q_{\{ j \} } m^{(j)}  } \, \abs{\det P_{\rr / \{ j \}} \IM_{\rr / \{ j \}}} 
 =\norm{m^{(j)}} \, \abs{\det P_{\rr / \{ j \}} \IM_{\rr / \{ j \}}} 
\end{equation}
finally proving our assertion for $M_r (\IM; u)$.

Now let us study $E_r (\IM; u) =   \displaystyle \int\displaylimits_{\IR^r} \ddd^r  \up \,  e^{- \pi (u - \up)^\tp (u - \up)} \,
\sign{\IM^\tp \up}$:
\begin{align}
&w^{(j) \tp} \del_u E_r (\IM; u)   
\notag \\
&\qquad =
\displaystyle \int\displaylimits_{\IR^r} \ddd^r  \up \, 
\lb - w^{(j) \tp} \del_\up  e^{- \pi (u - \up)^\tp (u - \up)} \rb \,
\sign{\IM_{\rr / \{ j \}}^\tp \up} \sign{m^{(j) \tp} \up}.  
\end{align}
Noting that $w^{(j) \tp} \del_\up \lp \IM_{\rr / \{ j \}}^\tp \up \rp = 0$ and $w^{(j) \tp} \del_\up \lb \sign{m^{(j) \tp} \up} \rb = 2 \d \lp m^{(j) \tp} \up \rp = 
\frac{2}{\norm{m^{(j)}}  } \,  \d \lp  Q_{\{ j \} }  \up \rp$ and integrating by parts we get
\begin{align}
w^{(j) \tp} \del_u E_r (\IM; u)  &=  \frac{2}{\norm{m^{(j)}}  }
\displaystyle \int\displaylimits_{\IR^r} \ddd^r  \up \, e^{- \pi (u - \up)^\tp (u - \up)} \, 
\sign{\IM_{\rr / \{ j \}}^\tp \up}  \d \lp  Q_{\{ j \} }  \up \rp     \\
&=  \frac{2}{\norm{m^{(j)}}  }
\displaystyle \int\displaylimits_{\IR^r} \ddd^r  \up \,
 e^{- \pi (u - \up)^\tp 
 \lp P_{\rr / \{ j \}}^\tp P_{\rr / \{ j \}} + Q_{ \{ j \}}^\tp Q_{\{ j \}} \rp 
 (u - \up)} \,   \notag \\
 & \qquad \times
\sign{\IM_{\rr / \{ j \}}^\tp
 \lp P_{\rr / \{ j \}}^\tp P_{\rr / \{ j \}} + Q_{ \{ j \}}^\tp Q_{\{ j \}} \rp 
  \up}  
\d \lp  Q_{\{ j \} }  \up \rp.
\end{align}
Performing a change of variables $\wt{u}'  = P_{\rr / \{ j \}} u$, $u_0 = Q_{\{ j \}} u$ and performing the integral over $u_0$ we can rewrite $w^{(j) \tp} \del_u E_r (\IM; u) $ as 
\begin{equation}
\frac{2}{\norm{m^{(j)}}  }
\displaystyle \int\displaylimits_{\IR^{r-1}} \ddd^{r-1}  \wt{u}'
 \, e^{- \pi (P_{\rr / \{ j \}} u - \wt{u}' )^\tp (P_{\rr / \{ j \}} u - \wt{u}')} \, 
 e^{- \pi u^\tp Q_{\{ j \} }^\tp Q_{\{ j \} } u    } \,
\sign{\IM_{\rr / \{ j \}}^\tp P_{\rr / \{ j \}}^\tp \wt{u}'}   
\end{equation}
which finally can be written as
\begin{equation}
\frac{2}{\norm{m^{(j)}}  } \, 
e^{- \pi u^\tp Q_{\{ j \} }^\tp Q_{\{ j \} } u    }  \,
E_{r-1} (P_{\rr / \{ j \} }  \IM_{\rr / \{ j \} }; P_{\rr / \{ j \} } u)
\end{equation}
proving our assertion for $E_r (\IM; u)$.

\end{proof}

\begin{prop}\label{prop:M_E_shadow}
The shadows of $E_r$ and $M_r$ are given by
\begin{equation}
\frac{i}{4} u^\tp \del_u M_r (\IM; u) = \frac{i}{2}
\sum_{j=1}^r   \frac{m^{(j) \tp} u }{\norm{m^{(j)}}  }
\, e^{- \pi u^\tp Q_{\{ j \} }^\tp Q_{\{ j \} } u    }  \,
M_{r-1} (P_{\rr / \{ j \} }  \IM_{\rr / \{ j \} }; P_{\rr / \{ j \} } u),
\end{equation}
and
\begin{equation}
\frac{i}{4} u^\tp \del_u E_r (\IM; u) = \frac{i}{2}
\sum_{j=1}^r   \frac{m^{(j) \tp} u }{\norm{m^{(j)}}  }
\, e^{- \pi u^\tp Q_{\{ j \} }^\tp Q_{\{ j \} } u    } \,
E_{r-1} (P_{\rr / \{ j \} }  \IM_{\rr / \{ j \} }; P_{\rr / \{ j \} } u).
\end{equation}
\end{prop}
\begin{proof}
Using the fact that $\IW^\tp \IM = \IM \IW^\tp = I_r$ we have 
\begin{equation}
u^\tp \del_u M_r (\IM; u) = 
u^\tp  \IM \IW^\tp \del_u M_r (\IM; u) = \sum_{j=1}^r \lp m^{(j) \tp} u \rp \, w^{(j) \tp} \del_u M_r (\IM; u).
\end{equation}
Using Lemma \ref{prop:M_E_derivative} then proves our statement for $M_r$. The proof for $E_r$ is exactly the same.
\end{proof}

\begin{prop}\label{prop_M_E_Vigneras}
$M_r (\IM; u)$ and $E_r (\IM; u)$ solve Vign\'eras equation with $\l = 0$ for quadratic form $Q(u) = u^\tp u$ on their domain of definition. In other words,
\begin{equation}
\sum_{j=1}^r \lp \del_{u_j}^2 + 2 \pi u_j \del_{u_j} \rp M_r (\IM; u) = 0
\andd
\sum_{j=1}^r \lp \del_{u_j}^2 + 2 \pi u_j \del_{u_j} \rp E_r (\IM; u) = 0.
\end{equation}
\end{prop}
\begin{proof}
We start with the case for $M_r (\IM; u)$. Since 
\begin{equation}
\lp \del_{u_j}^2 + 2 \pi u_j \del_{u_j} \rp 
e^{-\pi z^\tp z - 2 \pi i z^\tp u} = 2 \pi z_j \del_{z_j}  e^{-\pi z^\tp z - 2 \pi i z^\tp u}
\end{equation}
we get
\begin{align}
&\sum_{j=1}^r \lp \del_{u_j}^2 + 2 \pi u_j \del_{u_j} \rp M_r (\IM; u) 
\notag \\
&\quad\qquad =
 \lp \frac{i}{\pi} \rp^r  \abs{\det \IM}^{-1}  \int\displaylimits_{\IR^r - i u} \ddd^r z \,
\frac{2 \pi}{\prd \lp \IW^\tp z \rp}
\sum_{j=1}^r  z_j \del_{z_j}
e^{-\pi z^\tp z - 2 \pi i z^\tp u}.
\end{align}
Integrating by parts and noting that $\sum_{j=1}^r   \del_{z_j} \lp \frac{z_j}{\prd \lp \IW^\tp z \rp} \rp = 0$ proves the statement for $M_r$.

For $E_r(\IM; u)$ we first note that by Proposition \ref{prop:M_E_shadow} we have
\begin{equation}\label{eq:prop_M_E_Vigneras}
2 \pi  u^\tp \del_u E_r (\IM; u) = 4 \pi
\sum_{j=1}^r   \frac{m^{(j) \tp} u }{\norm{m^{(j)}}  }
\, e^{- \pi u^\tp Q_{\{ j \} }^\tp Q_{\{ j \} } u    } \,
E_{r-1} (P_{\rr / \{ j \} }  \IM_{\rr / \{ j \} }; P_{\rr / \{ j \} } u).
\end{equation}
Next we note that by Lemma \ref{prop:M_E_derivative} we have
\begin{align}
\sum_{j=1}^r \del_{u_j} &\del_{u_j} E_r (\IM; u) =
\sum_{j=1}^r m^{(j) \tp} \del_u \lp  w^{(j) \tp} \del_u \, E_r (\IM; u)  \rp \\
&= \sum_{j=1}^r m^{(j) \tp} \del_u  \lp  
\frac{2}{\norm{m^{(j)}}  }
e^{- \pi u^\tp Q_{\{ j \} }^\tp Q_{\{ j \} } u    } 
E_{r-1} (P_{\rr / \{ j \} }  \IM_{\rr / \{ j \} }; P_{\rr / \{ j \} } u)
\rp.
\end{align}
Since $P_{\rr / \{ j \} } m^{(j)} = 0$ we only need 
\begin{equation}
m^{(j) \tp} \del_u  \lp
e^{- \pi u^\tp Q_{\{ j \} }^\tp Q_{\{ j \} } u    } \rp =  - 2 \pi \lp m^{(j) \tp} u \rp
e^{- \pi u^\tp Q_{\{ j \} }^\tp Q_{\{ j \} } u    }
\end{equation}
to find
\begin{align}
\sum_{j=1}^r \del_{u_j} \del_{u_j} &E_r (\IM; u)
\notag \\
&= - 4 \pi
\sum_{j=1}^r   \frac{m^{(j) \tp} u }{\norm{m^{(j)}}  }
\, e^{- \pi u^\tp Q_{\{ j \} }^\tp Q_{\{ j \} } u    } \,
E_{r-1} (P_{\rr / \{ j \} }  \IM_{\rr / \{ j \} }; P_{\rr / \{ j \} } u)
\end{align}
canceling the contribution from equation \eqref{eq:prop_M_E_Vigneras} and proving $\sum_{j=1}^r \lp \del_{u_j}^2 + 2 \pi u_j \del_{u_j} \rp E_r (\IM; u)$ vanishes. Note that the same proof can also be used for $M_r(\IM; u)$ giving a second proof for the statement for $M_r(\IM; u)$.
\end{proof}

\begin{prop}\label{prop_M_bound}
$M_r (\IM; u)$ is uniformly bounded as $\abs{  M_r (\IM; u) } \leq (r!) \, e^{- \pi u^\tp u}$.
\end{prop}
\begin{proof}
We will use induction to prove the statement. For $r=1$ we have $\abs{M_1 (\IM; u)} = \abs{\mathrm{erfc} \lp \abs{u} \sqrt{\pi} \rp} \leq e^{-\pi u^2}$ establishing the base case. Now we assume the hypothesis holds for $M_{r-1}$ to prove the inductive step. By Proposition \ref{prop:M_E_shadow} we have
\begin{equation}
\frac{\ddd}{\ddd t} M_r (\IM; t u) = 
2 \sum_{j=1}^r   \frac{m^{(j) \tp} u }{\norm{m^{(j)}}  }
\, e^{- \pi t^2 u^\tp Q_{\{ j \} }^\tp Q_{\{ j \} } u    }  \,
M_{r-1} (P_{\rr / \{ j \} }  \IM_{\rr / \{ j \} }; t P_{\rr / \{ j \} } u).
\end{equation}
Integrating from $t=1$ to $\infty$ and noting that $M_r (\IM; t u) \to 0$ as $t \to \infty$ we have
\begin{equation}
M_r (\IM;  u) = 
2 \sum_{j=1}^r   \frac{m^{(j) \tp} u }{\norm{m^{(j)}}  }
\, \int_1^\infty \ddd t \, e^{- \pi t^2 u^\tp Q_{\{ j \} }^\tp Q_{\{ j \} } u    }  \,
M_{r-1} (P_{\rr / \{ j \} }  \IM_{\rr / \{ j \} }; t P_{\rr / \{ j \} } u).
\end{equation}
By the induction hypothesis
\begin{align}
\abs{  M_r (\IM; u) } &\leq 2 (r-1)!  \sum_{j=1}^r  \abs{ \frac{m^{(j) \tp} u }{\norm{m^{(j)}}  } }
\, \int_1^\infty \ddd t \, e^{- \pi t^2 u^\tp Q_{\{ j \} }^\tp Q_{\{ j \} } u    }
e^{- \pi t^2 u^\tp P_{\rr / \{ j \} }^\tp P_{\rr / \{ j \} } u    }   \\
&= (r-1)!  \sum_{j=1}^r \frac{\abs{Q_{\{ j \} } u }}{\sqrt{u^\tp u}} \, 
		\mathrm{erfc} \lp \sqrt{\pi u^\tp u} \rp.
\end{align}
Using $\frac{\abs{Q_{\{ j \} } u }}{\sqrt{u^\tp u}} \leq 1$ and $\mathrm{erfc} \lp \sqrt{\pi u^\tp u} \rp \leq e^{-\pi u^\tp u}$ then gives the result.
\end{proof}

\begin{prop}\label{prop:E_decomposition_of_M}
On its domain of definition (which is $\prd \lp \IW^\tp u \rp \neq 0$) the function $M_r (\IM; u)$  can be decomposed as
\begin{equation}
M_r (\IM; u) = \sum_{S \sbst \rr } (-1)^{r-\abs{S}} \, 
\sign{\IW_{\rr / S}^\tp u} \, E_{\abs{S}} \lp Q_S \IM_S; Q_S u \rp.
\end{equation}
\end{prop}
\begin{proof}
We start with a change of variables $v_j = w^{(j) \tp} z$ (so that $\abs{\frac{\del v}{\del z}} = \abs{\det \IW} = \abs{\det \IM}^{-1}$) in the integral defining $M_r (\IM; u)$ (see equation \eqref{eq:def_Mr}) to get
\begin{equation}
M_r (\IM; u) = 
 \lp \frac{i}{\pi} \rp^r    \int\displaylimits_{\IR^r - i \IW^\tp u} \ddd^r v \,
\frac{e^{-\pi v^\tp \IM^\tp \IM v - 2 \pi i u^\tp \IM v}}{\prod_{j=1}^r v_j}.
\end{equation} 
Now we deform the integration contours without crossing any poles:
\begin{equation}
M_r (\IM; u) = 
 \lp \frac{i}{\pi} \rp^r    \lim\displaylimits_{\e_j \to 0^+}  \int\displaylimits_{\IR^r} \ddd^r v \,
\frac{e^{-\pi v^\tp \IM^\tp \IM v - 2 \pi i u^\tp \IM v}}{\prod_{j=1}^r \lp  v_j  - i \e_j \,  \sign{w^{(j) \tp} u} \rp}.
\end{equation}
Under the integral and in the limit $\e_j \to 0^+$ we can replace $\frac{1}{v_j  - i \e_j \,  \sign{w^{(j) \tp} u}}$ factors with $\pv \lp \frac{1}{v_j} \rp + i \pi \, \sign{w^{(j) \tp} u} \, \d(v_j)$ and rewrite $M_r (\IM; u)$ as
\begin{align}
 \lp \frac{i}{\pi} \rp^r  &\int\displaylimits_{\IR^r} \ddd^r v \, 
 e^{-\pi v^\tp \IM^\tp \IM v - 2 \pi i u^\tp \IM v} 
 \notag \\
&\qquad
\times  \sum_{S \sbst \rr} \lp \prod_{j \in S} \pv \lp \frac{1}{v_j} \rp  \rp 
  \lp   \prod_{j \in \rr / S} \lb  i \pi \, \sign{w^{(j) \tp} u} \, \d(v_j) \rb  \rp.
\end{align}
Taking the integrals over $v_{\rr / S}$ using the delta functions we then have
\begin{equation}
\sum_{S \sbst \rr} (-1)^{r-\abs{S}} \,  \sign{\IW_{\rr / S}^\tp u} \,
\lp \frac{i}{\pi} \rp^{\abs{S}} \int\displaylimits_{\IR^{\abs{S}}} \ddd^{\abs{S}} v_S \,
\prod_{j \in S} \pv \lp \frac{1}{v_j} \rp \,
 e^{-\pi v_S^\tp \IM_S^\tp \IM_S v_S - 2 \pi i u^\tp \IM_S v_S}.
\end{equation}
Using the fact that $\frac{i}{\pi} \pv \lp \frac{1}{k} \rp$ is the Fourier transform of $\sign{x}$ and that $Q_S^\tp Q_S \IM_S = \IM_S$ we can rewrite the integral $\displaystyle \lp \frac{i}{\pi} \rp^{\abs{S}} \int\displaylimits_{\IR^{\abs{S}}} \ddd^{\abs{S}} v_S \,
\prod_{j \in S} \pv \lp \frac{1}{v_j} \rp \,
 e^{-\pi v_S^\tp \IM_S^\tp \IM_S v_S - 2 \pi i u^\tp \IM_S v_S}$ as
\begin{equation}
\int\displaylimits_{\IR^{\abs{S}}} \ddd^{\abs{S}} x
\int\displaylimits_{\IR^{\abs{S}}} \ddd^{\abs{S}} v_S \, \sign{x} \, e^{2 \pi i x^\tp v_S} \,
 e^{-\pi v_S^\tp \IM_S^\tp Q_S^\tp Q_S \IM_S v_S - 2 \pi i u^\tp Q_S^\tp Q_S \IM_S v_S}.
\end{equation}
Performing a change of variables $\wt{v}_S = Q_S \IM_S v_S$ and $\wt{x} = (Q_S \IM_S)^{- \tp} x$ (for which the Jacobian is unity) we obtain
\begin{equation}
\int\displaylimits_{\IR^{\abs{S}}} \ddd^{\abs{S}} \wt{x} 
\int\displaylimits_{\IR^{\abs{S}}} \ddd^{\abs{S}} \wt{v}_S \  \sign{\IM_S^\tp Q_S^\tp \wt{x}} \, 
e^{2 \pi i \wt{x}^\tp \wt{v}_S} \,
 e^{-\pi \wt{v}_S^\tp \wt{v}_S} \, e^{- 2 \pi i u^\tp Q_S^\tp \wt{v}_S}.
\end{equation}
Finally performing the Gaussian integral over $\wt{v}_S$ we find
\begin{equation}
\int\displaylimits_{\IR^{\abs{S}}} \ddd^{\abs{S}} \wt{x}  \ 
\sign{\IM_S^\tp Q_S^\tp \wt{x}} \, e^{- \pi (Q_S u - \wt{x})^\tp (Q_S u - \wt{x})} = E_{\abs{S}} \lp Q_S \IM_S; Q_S u \rp
\end{equation}
finishing the proof.
\end{proof}

The decomposition given in Proposition \ref{prop:E_decomposition_of_M} implies that one can conversely decompose $E_r (\IM; u)$ in terms of $M_r$ functions. Before giving this result we state a lemma that we will need in establishing that decomposition.

\begin{lem}\label{lem:signature_lemma}
For any $n \times n$ real positive definite matrix $G$ and any $v \in \IR^{n \times 1}$ such that 
$\displaystyle\prod_{S \sbst [n] } \lp \prd  \lb \mat{- G_{S,S}^{-1} & 0 \\ - G_{[n]/S,S}  \, G_{S,S}^{-1} & I_{n-\abs{S}}}   \mat{v_S \\ v_{[n]/S}} \rb \rp  \neq 0$ we have
\begin{equation}
\sum_{S \sbst [n] } \sign{  \mat{- G_{S,S}^{-1} & 0 \\ - G_{[n]/S,S} \,  G_{S,S}^{-1} & I_{n-\abs{S}}}   \mat{v_S \\ v_{[n]/S}}  } = 0.
\end{equation}
\end{lem}
\begin{proof}
See Appendix \ref{appendixA} for the proof.
\end{proof}

Taking $G_{S,T} = \IW_S^\tp \IW_T$ and $v_S = \IW_S^\tp u$ in Lemma \ref{lem:signature_lemma} and noting that 
\begin{equation}
\IM_S^\tp P_S^\tp P_S u = \lp  \IW_S^\tp \IW_S \rp^{-1}  \IW_S^\tp u,
\end{equation}
and
\begin{equation}
\IW_{N / S}^\tp Q_{\rr / S}^\tp Q_{\rr / S} u  = \IW_{N / S}^\tp u  - \lp \IW_{N / S}^\tp \IW_S \rp \lp  \IW_S^\tp \IW_S \rp^{-1}  \IW_S^\tp u
\end{equation}
we find 
\begin{equation}\label{eq:sign_lemma}
\sum_{S: S \sbst N} (-1)^{\abs{S}}  \, \sign{\IM_S^\tp P_S^\tp P_S u} \, 
\sign{\IW_{N / S}^\tp Q_{\rr / S}^\tp Q_{\rr / S} u}  = 0
\end{equation}
for all non-empty subsets $N$ of $\rr$ and for all $u$ such that the arguments of $\mathrm{sign}$ functions are non-zero.

\begin{prop}\label{prop:M_decomposition_of_E}
For any $u$ such that $\displaystyle\prod_{S \sbst [r] } \lb \prd \mat{ \IW_S^\tp Q_S^\tp Q_S u  \\ \IM_{\rr / S}^\tp P_{\rr / S}^\tp P_{\rr / S} u  } \rb \neq 0  $ we have
\begin{equation}
E_r (\IM; u) = \sum_{S \sbst \rr } 
\sign{\IM_{\rr / S}^\tp P_{\rr / S}^\tp P_{\rr / S} u} \, M_{\abs{S}} \lp Q_S \IM_S; Q_S u \rp.
\end{equation}
\end{prop}
Before going into the proof note that $M$ functions have discontinuities
\begin{equation}
M_{\abs{S}} \lp Q_S \IM_S; Q_S u \rp
\to 
(-1)^{\abs{S} - \abs{S'}} \,
\sign {\IW^\tp_{S/S'} Q_S^\tp Q_S u} \, M_{\abs{S'}} \lp Q_{S'} \IM_{S'}; Q_{S'} u \rp
\end{equation}
as $\IW^\tp_{S/S'} Q_S^\tp Q_S u \to 0$ where $S' \sbst S \sbst \rr$ by Proposition \ref{prop:Mr_discontunity}. So these discontinuities cancel if 
\begin{equation}
\sum_{S: N \sbst S \sbst \rr} (-1)^{\abs{S} - \abs{N}} \,
\sign {\IW^\tp_{S/N} Q_S^\tp Q_S u} \, \sign{\IM_{\rr / S}^\tp P_{\rr / S}^\tp P_{\rr / S} u} =0
\end{equation}
for all proper subsets $N$ of $\rr$. This in turn is ensured by Lemma \ref{lem:signature_lemma} and equation \eqref{eq:sign_lemma}. We now turn to the proof of Proposition \ref{prop:M_decomposition_of_E} to show the decomposition precisely:
\begin{proof}
Using Proposition \ref{prop:E_decomposition_of_M} we find $\displaystyle 
\sum_{S \sbst \rr } 
\sign{\IM_{\rr / S}^\tp P_{\rr / S}^\tp P_{\rr / S} u} \, M_{\abs{S}} \lp Q_S \IM_S; Q_S u \rp$
is equal to:
\begin{align}
\sum_{S: S \sbst \rr }  \sum_{N: N \sbst S } 
(-1)^{\abs{S} - \abs{N}} \,
&\sign {\IW^\tp_{S/N} Q_S^\tp Q_S u} \, \sign{\IM_{\rr / S}^\tp P_{\rr / S}^\tp P_{\rr / S} u}
\notag \\
&\qquad \times
E_{\abs{N}}  \lp Q_N \IM_N; Q_N u \rp.
\end{align}
Changing the order of sums gives
\begin{align}
\sum_{N: N \sbst \rr } & E_{\abs{N}}  \lp Q_N \IM_N; Q_N u \rp 
\notag \\
&\qquad \times
\sum_{S: N \sbst S \sbst \rr} 
(-1)^{\abs{S} - \abs{N}} \,
\sign {\IW^\tp_{S/N} Q_S^\tp Q_S u} \, \sign{\IM_{\rr / S}^\tp P_{\rr / S}^\tp P_{\rr / S} u} .
\end{align}
Then the sum over $S$ is zero by Lemma \ref{lem:signature_lemma} and equation \eqref{eq:sign_lemma} except for the case $N = \rr$ where it is unity. This then simply leaves $E_r ( \IM; u)$.
\end{proof}

\subsection{Boosted Error Functions}
We can now use the functions $E_s (\IM; u)$ and $M_s (\IM;u)$ we defined for Euclidean bilinear form in the previous section to spaces with arbitrary non-degenerate bilinear forms. In particular let $x \in \IR^n$ and let us define a signature $(r,n-r)$ bilinear form on this space by $B(x,y) = x^\tp A y$ (or by the associated quadratic form $Q(x) = x^\tp A x$). Here $r$ denotes the number of positive definite directions. We will define $E_r$ and $M_r$ functions using vectors $c_j \in \IR^n$ for $j=1,\ldots,s \leq r$ (represented as column vectors) which span a positive-definite subspace, in other words $C^\tp A C > 0$ where $C \equiv \lp c_1 \ldots c_r \rp$.

Let us introduce some notation before proceeding any further. Let $E \in \IR^{s \times n}$ be a matrix whose rows form an orthonormal basis for the plane spanned by $c_j$'s so that $E A E^\tp = I_s$ and $C = E^\tp E A C$. The projection of $x$ to the plane spanned by $c_j$'s will be denoted as $x^C_+ = E^\tp E A x = C (C^\tp A C)^{-1} C A x$. 

\begin{defn}
Let $A$ be a signature $(r,n-r)$ bilinear form and $C = (c_1 \ldots c_s)$ be an $n \times s$ matrix whose columns form a positive definite subspace according to this bilinear form. Further define a matrix $E \in \IR^{s \times n}$ whose rows define an orthonormal basis for the subspace spanned by $c_j$'s. then we define $E^A_s (C; x)$ and $M^A_s (C; x)$ as
\begin{equation}
E^A_s (C; x) = E_s (EAC; EAx) 
\andd
M^A_s (C; x) = M_s (EAC; EAx) .
\end{equation}
We will drop the superscript $A$ whenever the bilinear form is implied from the context and will drop the subscript $s$ when the number of vectors in $C$ can be inferred. 
\end{defn}
Note that these functions do not depend on the choice of $E$ since different choices correspond to a transformation $E \to QE$ where $Q \in \O (s, \IR)$ and that leaves $E^A_s (C; x)$ and $M^A_s (C; x)$ invariant by Proposition \ref{prop:Mr_basic_property}.

We also define $D = (d_1 \ldots d_s) \in \IR^{n \times s}$ whose columns form a dual basis to $C$ for the subspace $c_i$'s span. That is $D^\tp A C = I_s$ and $D = E^\tp E A D$ which can be easily verified for $D =E^\tp (EAC)^{- \tp}$. We also use $C_S$ for $S \sbst [s]$ to denote the matrix $C_S = \lp c_{j_1} \ c_{j_2} \ldots c_{j_{\abs{S}}} \rp$ where $j_1, j_2, \ldots, j_{\abs{S}} \in S$ and $j_1 < j_2 < \ldots < j_{\abs{S}}$. The matrix $D_S$ for the dual basis vectors is similarly defined. One last notation we will use is $C_{S \perp S'}$ denoting the projection of vectors in $C_S$ to the subspace orthogonal to the one spanned by $C_{S'}$.\footnote{In Section \ref{sec:indef_theta} we use the same notation also when the columns of $C_S$ span an indefinite signature subspace.} More concretely $C_{S \perp S'}$ will be formed by vectors $c_j - C_{S'} \lp C_{S'}^\tp A C_{S'} \rp^{-1} C_{S'}^\tp A \, c_j$ for $j \in S$ in increasing $j$ order though this choice will not be important. Now we can state the following propositions following from our work in Section \ref{sec:error_fnc}.

\begin{prop}\label{prop:E_boost_property}
\hfill
\begin{enumerate}[label=(\alph*)]
\item $E (C; x)$ is a $\mathcal{C}^\infty$ function of $x \in \IR^n$ for $C \in \IR^{n \times s}$ spanning a timelike subspace as described above. It is invariant under permutations of $c_j$, independent positive scalings of $c_j$ and is odd under independent sign flips of $c_j$'s.
\item If $C$ splits into two sets $C_1$ and $C_2$ which span orthogonal subspaces, then
\begin{equation}
E (C; x) = E (C_1;x) \,  E (C_2;x).
\end{equation}
\item As $\abs{B(c_j, x)} \to \infty$ for all $j$ we have $E (C; x) \to \sign{B(C,x)}$.
\item The function $E(C;x)$ satisfies the Vign\'eras equation for bilinear form $B(x,y)$ with $\l = 0$:
\begin{equation}
\lb B^{-1} (\del_x, \del_x ) + 2 \pi x^\tp \del_x \rb E(C;x) = 0
\end{equation}
where $ B^{-1}$ denotes $ B^{-1} (x, y ) = x^\tp A^{-1} y$. The shadow of $E(C;x)$ is 
\begin{equation}
\frac{i}{4} x \del_x E(C;x) = \frac{i}{2} \sum_{j=1}^{s} 
\frac{B(c_j, x)}{\sqrt{Q(c_j)}} \, e^{- \pi B(c_j,x)^2 / Q(c_j)} \, E(C_{[s]/\{j\} \perp \{j\} } ; x).
\end{equation}
\item The function $E(C;x)$ has an integral representation:
\begin{equation}
E(C;x) = \int\displaylimits_{\langle C \rangle}  \ddd^{s} x' \, e^{- \pi Q(x_+^C - x')} \,
\sign{B(C,x')}
\end{equation}
where the measure is normalized so that $\displaystyle\int\displaylimits_{\langle C \rangle}  \ddd^{s} x' \, e^{- \pi Q(x')} = 1$.
\end{enumerate}
\end{prop}

\begin{prop}\label{prop:M_boost_property}
\hfill
\begin{enumerate}[label=(\alph*)]
\item $M (C; x)$ is a $\mathcal{C}^\infty$ function of $x \in \IR^n$ away from the loci $B(d_j, x) = 0$ for $C \in \IR^{n \times s}$ spanning a timelike subspace. It is invariant under permutations of $c_j$, independent positive scalings of $c_j$ and is odd under independent sign flips of $c_j$'s.
\item If $C$ splits into two sets $C_1$ and $C_2$ which span orthogonal subspaces, then
\begin{equation}
M (C; x) = M (C_1;x) \,  M (C_2;x).
\end{equation}
\item $\abs{M (C;x)} < (s!) \, e^{- \pi Q(x_+^C)}$.
\item The function $M(C;x)$ satisfies the Vign\'eras equation for bilinear form $B(x,y)$ with $\l = 0$.
\item The function $M(C;x)$ has an integral representation:
\begin{equation}
M(C;x) = \lp \frac{i}{\pi} \rp^{s} \lp \det C^\tp A C \rp^{-1}
\int\displaylimits_{\langle C \rangle - i x_+^C}  \ddd^{s} z \, 
\frac{e^{- \pi Q(z) - 2 \pi i B(x,z)}}{\prod \lb B(D,z) \rb}
\end{equation}
where the measure is normalized so that $\displaystyle\int\displaylimits_{\langle C \rangle}  \ddd^{s} x' \, e^{- \pi Q(x')} = 1$.
\end{enumerate}
\end{prop}

\begin{prop}\label{prop:boosted_decomposition}
On its domain of definition (that is $\prod \lb B(D,z) \rb \neq 0$) we have the decomposition
\begin{equation}
M (C;x) = \sum_{S \sbst [s]} (-1)^{s- \abs{S}} \, \sign{ B(D_{[s]/S},x) }  \, E(C_S; x).
\end{equation}
Similarly we have
\begin{equation}
E (C;x) = \sum_{S \sbst [s]} \sign{ B(C_{[s]/S \perp S},x) }  \, M(C_S; x)
\end{equation}
for any $x$ such that the $M$ function is well defined and the arguments of $\mathrm{sign}$ functions are nonzero.
\end{prop}

\section{Indefinite Theta Functions of Higher Depth}\label{sec:indef_theta}
In this section we will construct a certain indefinite theta series and give sufficient conditions for its convergence. The holomorphic part of these series will be given by restricting the sum over lattice points through the function
\begin{equation}
\phi_r (x) = \frac{1}{2^r} \prod_{j=1}^r \lb 
\sign{B(c_j,x)} - \sign{B(c'_j,x)}
\rb .
\end{equation}
Before stating our result let us introduce some notation. By $C_{S^P}$ we will mean the matrix whose columns are taken from the set $\{ c_j \st j \in S \cap P \} \cup \{ c'_j \st j \in S/P \}$ in, say, increasing $j$ order (we will also use $C'$ for $C_{\rr^\emptyset}$ and $C^P$ for $C_{\rr^P}$). We also form the matrix $C_{S^P \perp T^Q}$ by which we mean the matrix formed by vectors in $C_{S^P}$ projected to the subspace orthogonal to the one spanned by the vectors in $C_{T^Q}$. Next we will use $\DD(x_1, \ldots, x_s)$ for the determinant of the Gram matrix for the vectors $x_1, \ldots, x_s$ and $D_{j_1, j_2}$ for the cofactor at $(j_1,j_2)$ position for the Gram matrix constructed from $\{ c_1, c'_1, \ldots, c_r, c'_r \}$ where we will use primes in the subscript to denote positions corresponding to vectors $c'_j$'s. Finally form the matrix $M$ from the cofactor matrix of the Gram matrix for $\lp c_1, c'_1, \ldots, c_r, c'_r \rp$ and by removing cofactors $D_{j,j'}$ and $D_{j',j}$ for all $j = 1, \ldots, r$. 

\begin{thm}\label{thm:indef_theta}
Let $C$ and $C'$ as described above be $2r$ vectors in $\IR^n$ endowed with an integral bilinear form $B(x,y)$ of signature $(r,n-r)$. Assume that each $C^P$ for $P \sbst \rr$ spans a signature $(r,0)$ (i.e. positive-definite) subspace. Further assume that $\DD \equiv \DD (C,C')$ satisfies $\DD (-1)^r > 0$ (signifying that $C \cup C'$ forms a linearly independent set and spans a signature $(r,r)$ subspace by our assumption above) and that $D_{j,j'} (-1)^r \geq 0$ for all $j = 1, \ldots, r$. Finally assume that the matrix $(-1)^r M$ as defined above is negative definite.

Then $\theta_\mu [\phi_r,0]$ is a convergent series and it is holomorphic in $\t$ and $z$ away from the loci where $B(k+b,c_j) = 0$ or $B(k+b,c'_j) = 0$ for some $j \in \rr$ and $k \in \LL + \mu + p/2$.

Moreover, assume that analogous conditions for vectors in $C$ and $C'$ we stated above also holds for $C_{\rr/S \perp S^P}$ and $C'_{\rr/S \perp S^P}$ for any $S \sbst \rr$ and $P \sbst S$. Then $\theta_\mu [\wh{\phi}_r,0]$ with the kernel
\begin{equation}
\wh{\phi}_r (x)= \frac{1}{2^r}  \sum_{P \sbst \rr} (-1)^{\abs{P}} \, E_r (C^P ; x)
\end{equation}
is a convergent series and forms a modular completion for $\theta_\mu [\phi_r,0]$ transforming like a (vector-valued) Jacobi form of weight $(n/2,0)$.
\end{thm}
\begin{proof}
We follow and generalize the proofs in \cite{zwegers2008mock} and \cite{Alexandrov:2016enp}. The first thing to note is that for any $x \in \IR^n$ we have
\begin{equation}
\DD (x, c_1, c'_1, \ldots, c_r, c'_r) = \DD \lb Q(x)  - 2 \frac{\sum_{j=1}^r  D_{j,j'} B(c_j,x) B(c'_j,x)  }{\DD} \rb - X^\tp M X
\end{equation}
where $X^\tp = \lp B(c_1,x) \ B(c'_1,x) \ldots  B(c_r,x) \ B(c'_r,x) \rp$. We define the part in brackets as $Q_- (x)$:
\begin{equation}
Q_- (x) \equiv Q(x)  - 2 \frac{\sum_{j=1}^r  D_{j,j'} B(c_j,x) B(c'_j,x)  }{\DD}. 
\end{equation}
Now we note that if $x$ is linearly independent from $C \cup C'$ the subspace $\langle x, c_1, c'_1, \ldots, c_r, c'_r \rangle$ has signature $(r,r+1)$ and hence 
\begin{equation}
(-1)^r \DD (x, c_1, c'_1, \ldots, c_r, c'_r)  = (-1)^r \DD Q_- (x) - X^\tp \lb (-1)^r M \rb X < 0.
\end{equation}
 Using the negative definiteness of $(-1)^r M$ and positivity of $(-1)^r \DD$ we conclude $Q_- (x) < 0$. On the other hand, if $x$ is in the plane spanned by $C \cup C'$ we have $\DD (x, c_1,  \ldots, c'_r) = 0$ and hence $Q_- (x) < 0$ unless $x=0$ again arguing through negative definiteness of $(-1)^r M$ and positivity of $(-1)^r \DD$. Now, $\phi_r (x) \neq 0$ only when $\sign{B(c_j,x)} \,  \sign{B(c'_j,x)} \leq 0$ for all $j = 1, \ldots, r$. The assumptions $D_{j,j'} (-1)^r \geq 0$ and $\DD (-1)^r > 0$ implies $Q_- (x) \geq Q(x)$, i.e. on the support of $\phi_r (x)$ the negative definite bilinear form $Q_- (x)$ dominates $Q(x)$. Using this we can conclude that $\phi_r (x) e^{\pi Q(x) /2 } \leq e^{\pi Q_-(x) /2 }$ proving the absolute convergence of $\theta_\mu [\phi_r,0]$.
 
 For the second part of the theorem we use the decomposition in Proposition \ref{prop:boosted_decomposition} to rewrite the kernel  $\wh{\phi}_r (x)$ as
 \begin{equation}
 \wh{\phi}_r (x) = \frac{1}{2^{\abs{S}}} \sum_{S \sbst \rr} \sum_{P \sbst S} (-1)^{\abs{P}}
 \, M(C_{S^P};x) \,  \frac{1}{2^{r - \abs{S}}} \sum_{Q \sbst \rr/S} (-1)^{\abs{Q}} \, \sign{B(C_{\rr/S^Q \perp S^P},x)}.
 \end{equation}
 Let us focus on each $S \in \rr$ and $P \sbst S$ contribution
 \begin{equation}
  \phi_{S^P} (x) \equiv M(C_{S^P};x) \,  \lp \frac{1}{2^{r - \abs{S}}} \sum_{Q \sbst \rr/S} (-1)^{\abs{Q}} \, \sign{B(C_{\rr/S^Q \perp S^P},x)} \rp
\end{equation}  
separately. We start by decomposing each $x \in \IR^n$ as $x = x_1 + x_2$ where $x_1$ is in the linear span of $C_{S^P}$ and $x_2$ is in its orthogonal complement so that $Q(x) = Q(x_1) + Q(x_2)$. That divides $\phi_{S^P} (x) \, e^{\pi Q(x)/2}$ into a factor along $\langle C_{S^P} \rangle^\perp$:
\begin{equation}
\lp \frac{1}{2^{r - \abs{S}}} \sum_{Q \sbst \rr/S} (-1)^{\abs{Q}} \, \sign{B(C_{\rr/S^Q \perp S^P},x_2)} \rp e^{\pi Q(x_2)/2}
\end{equation}
and a factor along $\langle C_{S^P} \rangle$:
\begin{equation}
M(C_{S^P};x_1) \, e^{\pi Q(x_1)/2} .
\end{equation}
By our argument in the first part and by the hypothesis given for $C_{\rr/S \perp S^P} \cup C'_{\rr/S \perp S^P}$ the factor
\begin{equation}
\lp \frac{1}{2^{r - \abs{S}}} \sum_{Q \sbst \rr/S} (-1)^{\abs{Q}} \, \sign{B(C_{\rr/S^Q \perp S^P},x)} \rp e^{\pi Q(x_2)/2}
\end{equation}
is dominated along $\langle C_{S^P} \rangle^\perp$
by $e^{\pi Q^{S^{P\perp}}_- (x_2) /2}$ where $Q^{S^{P\perp}}_-$ is a negative definite bilinear form on $\langle C_{S^P} \rangle^\perp$ and correspondingly by Proposition \ref{prop:M_boost_property} 
\begin{equation}
\abs{M(C_{S^P};x)} e^{\pi Q(x_1)/2} \leq \abs{S}! \,  e^{-\pi Q(x_1)/2}
\end{equation}
 and this contribution is exponentially suppressed along $\langle C_{S^P} \rangle$. That shows the series in $\theta_\mu [\wh{\phi}_r,0]$ is convergent and that $\wh{\phi}_r (x) e^{\pi Q(x)/2}$ satisfies the conditions given in theorem \ref{thm:vigneras}.  Moreover, since $E(C^P;x)$ functions each satisfies Vign\'eras equation with $\l=0$ (see Proposition \ref{prop:E_boost_property}) by Vign\'eras' theorem \ref{thm:vigneras}, $\theta_\mu [\wh{\phi}_r,0]$ transforms like a (vector-valued) Jacobi form of weight $(n/2,0)$.
\end{proof}

\begin{rem}
It is desirable to further relax and simplify the conditions we put on $C$. See \cite{kudla2016theta} and \cite{raum} for further discussion.
\end{rem}

Aside from the obvious factorizable solutions to the hypothesis we put for $C$ and $C'$ we will exhibit a non-factorizable example for $r=4$ case.

\textbf{Example.} Consider signature $(4,4)$ integral bilinear form
\begin{equation}
A = \mat{G(A_4) & - I_4 \\ - I_4 & 0}
\end{equation}
where $G(A_4)$ denotes the Gram matrix for the $A_4$ root lattice:
\begin{equation}
G(A_4) = \mat{2 & -1 & 0 & 0 \\ -1 & 2 & -1 & 0 \\ 0 &  -1 & 2 & -1 \\ 0 &  0 &  -1 & 2}.
\end{equation}
Then the vectors
\begin{equation}
c_1 = \pmat{1 \\ 0 \\ 0 \\ 0 \\ 0 \\ 0 \\ 0 \\ 0 }, \quad
c_2 = \pmat{0 \\ 1 \\ 0 \\ 0 \\ 0 \\ 0 \\ 0 \\ 0 }, \quad
c_3 = \pmat{0 \\ 0 \\ 1 \\ 0 \\ 0 \\ 0 \\ 0 \\ 0 }, \quad
c_4 = \pmat{0 \\ 0 \\ 0 \\ 1 \\ 0 \\ 0 \\ 0 \\ 0 }
\end{equation}
and
\begin{equation}
c'_1 = \pmat{1 \\ 0 \\ 0 \\ 0 \\ 0 \\ -1 \\ 0 \\ 0 }, \quad
c'_2 = \pmat{0 \\ 1 \\ 0 \\ 0 \\ 0 \\ 0 \\ -1 \\ 0 }, \quad
c'_3 = \pmat{0 \\ 0 \\ 1 \\ 0 \\ 0 \\ 0 \\ 0 \\ -1 }, \quad
c'_4 = \pmat{0 \\ 0 \\ 0 \\ 1 \\ -1 \\ 0 \\ 0 \\ 0 }
\end{equation}
satisfy the hypotheses of  theorem \ref{thm:indef_theta}.

\section{Discussion}\label{sec:Discussion}
In this work we studied the properties of $r$-tuple error functions and introduced indefinite theta series using these functions. One obvious question is to relax the conditions we imposed on $c_j$'s and $c'_{j}$'s that determine the subset of lattice points used in the holomorphic part of the associated theta series and ensure its convergence. Specifically, one would want to allow null vectors and allow linear dependencies, which is essential to extend the range of applications for indefinite theta functions. The other two constructions for mock modular forms given by \cite{zwegers2008mock}, namely using Appell-Lerch sums \cite{zwegers2010multivariable} and meromorphic Jacobi forms \cite{Dabholkar:2012nd}, are closely related to signature $(1,n-1)$ indefinite theta series. It is then natural to look for similar corresponding constructions for signature $(r,n-r)$ indefinite theta functions. On the side of Appell-Lerch sums one such generalization is already available in literature under the name `generalized Appell functions' \cites{Manschot:2014cca,bringmann2015identities}. Indeed, \cite{Alexandrov:2016enp} initiated the study of their modular properties for the $r=2$ case. To study the complete story it is then desirable to study the null limits of the construction we have given for higher $r$ cases.

\appendix
\section{}\label{appendixA}
In this section we are going to prove the Lemma \ref{lem:signature_lemma} which we restate here for reference.

\begin{lem}
For any $n \times n$ real positive definite matrix $G$ and any $v \in \IR^{n \times 1}$ such that 
$\displaystyle\prod_{S \sbst [n] } \lp \prd  \lb \mat{- G_{S,S}^{-1} & 0 \\ - G_{[n]/S,S}  \, G_{S,S}^{-1} & I_{n-\abs{S}}}   \mat{v_S \\ v_{[n]/S}} \rb \rp  \neq 0$ we have
\begin{equation}
\sum_{S \sbst [n] } \signb{  \mat{- G_{S,S}^{-1} & 0 \\ - G_{[n]/S,S}  \, G_{S,S}^{-1} & I_{n-\abs{S}}}   \mat{v_S \\ v_{[n]/S}}  } = 0.
\end{equation}
\end{lem}
\begin{proof}
We will use induction on $n$. The base hypothesis easily follows from the positivity of $G_{1,1}$. For the inductive step let us note the following facts first. There are $2^n \, n$ $\mathrm{sign}$ functions in our sum. We are going to show that there are generically $2^{n-1} \, n$ independent ones that each appear twice and that discontinuities cancel among each pair. In particular, we consider the contribution to the sum above from subsets $S$ and $\Si$ for some $j \in [n]$ and $S \sbst [n] / \{ j \}$ and single out the contribution from the row corresponding to $v_j$.

\begin{itemize}
\item The contribution from $S$ reads (using $\ws$ for $[n]/(\Si)$):
\begin{align}
&\signb{\mat{-G_{S,S}^{-1}  & 0  & 0 \\  
-G_{j,S} \, G_{S,S}^{-1} & 1 & 0  \\
-G_{\ws,S} \, G_{S,S}^{-1} & 0 & I_{\abs{\ws}}
} \mat{v_S \\ v_j \\ v_\ws}}
\notag \\
&\qquad =
\sign{v_j - G_{j,S} \, G_{S,S}^{-1} v_S} \ 
\signb{\mat{-G_{S,S}^{-1}  & 0   \\ 
-G_{\ws,S} \, G_{S,S}^{-1} & I_{\abs{\ws}}
} \mat{v_S \\ v_\ws}}.
\end{align}

\item For the contribution from $\Si$ first note that 
\begin{equation}
\mat{G_{S,S}  & G_{S,j}  \\ G_{j,S}  & G_{j,j}}^{-1}
= \mat{
\lp G_{S,S} - \frac{1}{G_{j,j}} G_{S,j} \, G_{j,S}  \rp^{-1}   & - \frac{1}{k} G_{S,S}^{-1} \, G_{S,j}  \\
- \frac{1}{k}  G_{j,S} \, G_{S,S}^{-1}  & \frac{1}{k}
}
\end{equation}
where $k = G_{j,j} - G_{j,S} \, G_{S,S}^{-1} \, G_{S,j}$ and 
\begin{equation}
\lp G_{S,S} - \frac{1}{G_{j,j}} G_{S,j} \, G_{j,S}  \rp^{-1} = G_{S,S}^{-1} + \frac{1}{k} G_{S,S}^{-1} \, G_{S,j} \, G_{j,S} \, 
G_{S,S}^{-1}.
\end{equation}
By the assumption that $G$ is positive definite we have $k>0$. We can rewrite the $\Si$ contribution
\begin{equation}
\signb{\mat{-G_{\Si,\Si}^{-1}  & 0   \\ 
-G_{\ws,\Si} G_{\Si,\Si}^{-1} & I_{\abs{\ws}}
} \mat{v_\Si \\ v_\ws}}
\end{equation}
as
\begin{align}
& \signb{
\pmat{
- \lp G_{S,S} - \frac{1}{G_{j,j}} G_{S,j} G_{j,S}  \rp^{-1} & \frac{1}{k} G_{S,S}^{-1} G_{S,j}  & 0\\
\frac{1}{k}  G_{j,S} G_{S,S}^{-1}  & - \frac{1}{k} & 0\\
-G_{\ws,S} \lp G_{S,S} - \frac{1}{G_{j,j}} G_{S,j} G_{j,S}  \rp^{-1} 
+ \frac{1}{k} G_{\ws,j} G_{j,S}  G_{S,S}^{-1} \ \ 
& \frac{1}{k}  G_{\ws,S} G_{S,S}^{-1} G_{S,j} - \frac{1}{k}G_{\ws,j} \  & 
I_{\abs{\ws}}
}
\mat{v_S \\ v_j \\ v_\ws}}  =  \\
& 
- \sign{v_j - G_{j,S} G_{S,S}^{-1} v_S} \ 
\signb{\pmat{-\lp G_{S,S} - \frac{1}{G_{j,j}} G_{S,j} G_{j,S}  \rp^{-1}  & 0   \\ 
- \lp G_{\ws,S} - \frac{1}{G_{j,j}} G_{\ws,j} G_{j,S}  \rp
\lp G_{S,S} - \frac{1}{G_{j,j}} G_{S,j} G_{j,S}  \rp^{-1} & I_{\abs{\ws}}
} \mat{\wt{v}_S \\ \wt{v}_\ws}}  \notag
\end{align}
where we defined $\displaystyle \wt{v}_S = v_S - \frac{v_j}{G_{j,j}} G_{S,j}$.

\end{itemize}
Next we are going to show that possible discontinuities due to $\sign{v_j - G_{j,S} \, G_{S,S}^{-1} v_S}$ terms do cancel between these two contributions. For this we note that at 
$v_j =  G_{j,S} \, G_{S,S}^{-1} v_S$ we have
\begin{align}
&- \lp G_{S,S} - \frac{1}{G_{j,j}} G_{S,j} G_{j,S}  \rp^{-1} \wt{v}_S 
\notag \\
&\qquad= - \lp G_{S,S}^{-1} + \frac{1}{k} G_{S,S}^{-1} \, G_{S,j} G_{j,S} \, G_{S,S}^{-1} \rp
\lp v_S - \frac{1}{G_{j,j}} G_{S,j} G_{j,S} \,  G_{S,S}^{-1} \, v_S \rp   \notag  \\
&\qquad= - G_{S,S}^{-1} v_S
\end{align}
and 
\begin{align}
- &\lp G_{\ws,S} - \frac{1}{G_{j,j}} G_{\ws,j} G_{j,S}  \rp
\lp G_{S,S} - \frac{1}{G_{j,j}} G_{S,j} G_{j,S}  \rp^{-1} \wt{v}_S  + \wt{v}_\ws \notag \\
&\qquad = - \lp G_{\ws,S} - \frac{1}{G_{j,j}} G_{\ws,j} G_{j,S}  \rp
G_{S,S}^{-1} \, v_S  + \lp v_\ws -  \frac{1}{G_{j,j}} G_{\ws,j} G_{j,S} \, G_{S,S}^{-1} \, v_S \rp \notag\\
&\qquad = - G_{\ws,S} \, G_{S,S}^{-1}\, v_S  +  v_\ws.
\end{align}
So at  $\displaystyle \prd  \lb \mat{- G_{S,S}^{-1} & 0 \\ - G_{\ws,S}  G_{S,S}^{-1} & I_{\abs{\ws}}}   \mat{v_S \\ v_{\ws}} \rb  \neq 0$ (ensured by the hypothesis) the sum
\begin{equation}\label{eq:sign_sum}
\sum_{S \sbst [n] } \sign{  \mat{- G_{S,S}^{-1} & 0 \\ - G_{[n]/S,S}  \, G_{S,S}^{-1} & I_{n-\abs{S}}}   \mat{v_S \\ v_{[n]/S}}  } 
\end{equation}
is equal on both sides of $v_j - G_{j,S} \, G_{S,S}^{-1} \, v_S = 0$. The argument generalizes for all the $\mathrm{sign}$ functions in the sum. 

Let us now specialize to $j=n$ (the choice of $j=n$ is not important), use $\ws = [n-1]/S$ and rewrite the sum in \eqref{eq:sign_sum} as
\begin{align}
\sum_{S \sbst [n-1] } &
\lb
 \signb{  \mat{- G_{S,S}^{-1} & 0 \\ - G_{\ws,S}  \, 
 G_{S,S}^{-1} \  & I_{\abs{\ws}}}   \mat{v_S \\ v_{\ws}}  } 
 -
  \signb{  \mat{- \wt{G}_{S,S}^{-1}  & 0 \\ - \wt{G}_{\ws,S} \,  \wt{G}_{S,S}^{-1}\  & I_{\abs{\ws}}}   \mat{\wt{v}_S \\ \wt{v}_{\ws}}  } 
 \rb   \notag \\
 & \qquad \times \sign{v_n - G_{n,S} \, G_{S,S}^{-1} \, v_S}
\end{align}
where 
$\wt{v}_S = v_S - \frac{v_n}{G_{n,n}} G_{S,n}$
and $\wt{G}$ is a positive definite matrix defined by
\begin{equation}
\wt{G}_{[n-1],[n-1]} \equiv G_{[n-1],[n-1]} - \frac{1}{G_{n,n}} G_{[n-1],n} G_{n,[n-1]}.
\end{equation}

Now for any $v$ satisfying the hypothesis, we start increasing $v_n$ while keeping $v_{[n-1]}$ fixed until $v_n - G_{n,S} \, G_{S,S}^{-1} \, v_S > 0$ for all $S \sbst [n-1]$ and $v$ satisfies the hypothesis of the lemma. The value of our sum does not change across any of the possible discontinuities by our argument above. The fact that the sum over $S \sbst [n-1] $ is zero by the induction hypothesis then proves the statement of the lemma.

\end{proof}


\end{document}